\documentclass[a4paper,reqno]{amsart}
\usepackage{amssymb,amsmath,amsthm,bbm,dsfont,mathrsfs,graphicx,subfigure}
\usepackage[all]{xy} 
\usepackage{pstricks,pstricks-add} 
\usepackage{epsfig}
\usepackage{hyperref}

\newtheorem{mthm}{Theorem}

\newtheorem{thm}{Theorem}[section]
 \newtheorem{cor}[thm]{Corollary}
 \newtheorem{lem}[thm]{Lemma}
 \newtheorem{prop}[thm]{Proposition}

 \theoremstyle{definition}
 \newtheorem{defn}[thm]{Definition}
 \theoremstyle{remark}
 \newtheorem{rem}[thm]{Remark}

 \newcommand{\C}{\mathbb{C} }
 \newcommand{\CC}{\widehat{\C} }
 \newcommand{\R}{\mathbb{R} }

 \newcommand{\D}{\mathbb{D} }
 \newcommand{\Z}{\mathbb{Z} }
 \newcommand{\Q}{\mathbb{Q} }
 \newcommand{\M}{\mathcal{M} }

 \newcommand{\Perp}{\perp \! \! \! \perp}

\begin{document}
 
\title[Constructing clusters with matings]{Constructing rational maps with cluster points using the mating operation}
\author{Thomas Sharland}
\address{University of Warwick, Coventry, CV4 7AL, UK}
\email{tomkhfc@hotmail.com}
\date{\today}

\begin{abstract}
In this article, we show that all admissible rational maps with fixed or period two cluster cycles can be constructed by the mating of polynomials. We also investigate the polynomials which make up the matings that construct these rational maps. In the one cluster case, one of the polynomials must be an $n$-rabbit and in the two cluster case, one of the maps must be either $f$, a ``double rabbit'', or $g$, a secondary map which lies in the wake of the double rabbit $f$. There is also a very simple combinatorial way of classifiying the maps which must partner the aforementioned polynomials to create rational maps with cluster cycles. Finally, we also investigate the multiplicities of the shared matings arising from the matings in the paper.
\end{abstract}
\maketitle

\tableofcontents

\section{Introduction}

The study of holomorphic dynamical systems was first given serious study by Pierre Fatou \cite{Fatou:1919,Fatou:1920} and Gaston Julia \cite{Julia:1918} in the earlier part of the 20th Century. After laying relatively dormant for a number years, the subject was given a new lease of life due to the improvements in technology which allowed mathematicians to view the various sets under investigation. In the mid-1980s, Hubbard and Douady produced their now famous ``Orsay lecture notes'' \cite{DouadyHubbard:Orsay1,DouadyHubbard:Orsay2}, and since then the subject has grown from strength to strength.

This paper is a partner to \cite{Thurstoneq}, in which it was shown that, in certain cases, a very simple set of combinatorial data can classify (in the sense of Thurson equivalence) a rational map with a periodic cluster cycle. Like its partner paper, this paper is made up of results found in the author's PhD thesis \cite{Mythesis}. Essentially, the former paper showed the uniqueness part of the classification of maps with cluster cycles, whereas the current paper is concerned with the existence part. We will show that all realisable combinatorial data for rational maps can be realised by matings of polynomials and so, when combined with the results of \cite{Thurstoneq}, we will show that all maps with cluster cycles are matings.

\subsection{Definitions}

Let $f \colon \CC \to \CC$ be a rational map on the Riemann sphere, of degree at least 2. The Julia set $J(f)$ will be the closure of the set of repelling periodic points of $f$, and the Fatou set is the set $F(f) = \CC \setminus J(f)$. The connected components of $F(f)$ are called Fatou components. In the case where the critical orbits are periodic, we call the immediate basins of the (super)attracting orbit critical orbit Fatou components. We concern ourselves in the main with bicritical rational maps with marked critical points and such that the two critical points belong to the attracting basins of two disjoint periodic orbits of the same period. If $f$ is a polynomial, then we define the filled Julia set $K(f)$ to be the set of points that are not attracted by the superattracting fixed point at infinity. It is well-known that $\partial K(f) = J(f)$. 

\begin{defn}
	Let $F \colon \CC \to \CC$ be a bicritical rational map. Then a cluster point for $F$ is a point in $J(F)$ which is the endpoint of the angle $0$ internal rays of at least one critical orbit Fatou component from each of the two critical cycles. We will define a cluster to be the union of the cluster point and the Fatou components meeting at it. The period of the cluster will be the period of the cluster point. The star of a cluster will be the union of the cluster point and the associated $0$ internal rays, including the points on the critical orbit.
\end{defn}

A standard example of a (fixed) cluster cycle is the rational map formed by the mating of the rabbit with the airplane, see Figure~\ref{f:rabair}. Notice that, by our restriction to bicritical maps, the rational maps in this paper can have at most one cluster cycle.  

\begin{figure}[ht]
    \begin{center}
    \includegraphics[width=0.95\textwidth]{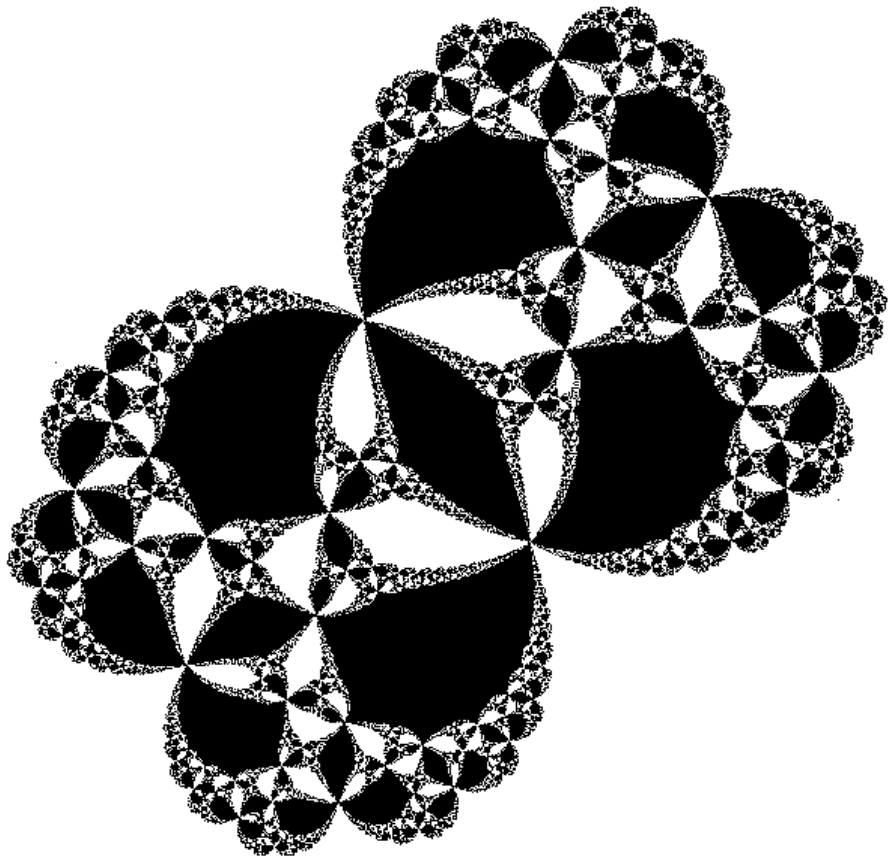}
    \end{center}
    \caption{A rational function with a fixed cluster point. This map is formed by the mating of the rabbit with the airplane.}
\label{f:rabair}
\end{figure}

We will define the critical displacement of a rational map with either a fixed cluster point or a period two cluster cycle. Informally, the critical displacement will give a (combinatorial) measure as to how far apart the two critical orbits are in the clusters. In actual fact, we require a different definition of this definition depending on whether we are studying the fixed case or the period two case. The reason for this is due to the result of the following lemma, which was proved in \cite{Thurstoneq}.

\begin{lem}\label{p:Marysresult}
	There does not exist a rational map $F$ with a period 2 cluster cycle such that the critical points are in the same cluster.
\end{lem}

The two definitions of critical displacement are the following. It is clear in both cases that the critical displacement is an odd integer, and that the definition is dependent on the labelling of the critical points of the map.

\begin{defn}
 Let $F$ be a rational map with a fixed cluster point. Label the endpoints of the star as follows. Let $e_0$ be the first critical point, and label the remaining arms in anticlockwise order by $e_1,e_2,\ldots,e_{2n-1}$. Then the second critical point is one of the $e_j$, and we call $j$ the critical displacement of the cluster of $F$. We denote the critical displacement by $\delta$.
\end{defn}

We will sometimes use the fact that the critical displacement can equally well be calculated as the combinatorial distance between the critical values, as opposed to the critical points. Also, notice that if one choice of marking for the critical points gives $\delta = k$, then the alternative marking will give a critical displacement of $\delta = -k$. Clearly, in light of Lemma~\ref{p:Marysresult}, any attempt to use the above definition to define the critical displacement in the period two case is impossible, hence we give the following definition for the period two case.

\begin{defn}
  Let $F$ be a rational map with a period two cluster cycle. Choose one of the critical points to be $c_1$, and label the cluster containing it to be $\mathcal{C}_{1}$. Then (by Lemma~\ref{p:Marysresult}) the other critical point $c_2$ is in the second cluster $\mathcal{C}_2$. We define the critical displacement $\delta$ as follows. Label the arms in the star of $\mathcal{C}_1$, starting with the arm with endpoint $c_1$, in anticlockwise order $\ell_0, \ell_1, \ldots, \ell_{2n-1}$. Then $F(c_2)$ is the endpoint of one of the $\ell_k$. This integer $k$ is the critical displacement, which we again denote by $\delta$.
\end{defn}

For the period two definition, we will later be using the fact that the critical displacement is also equal to the combinatorial distance between the $F(c_1)$ and $F^{\circ 2}(c_2)$. In contrast to the fixed case, there is not a neat symmetry to the two possible critical displacements a rational map can have. We will discuss this in further detail in Section~\ref{s:shared}.

\begin{defn}\label{d:CRN}
    Let $F \colon \CC \to \CC$ be a bicritical rational map and let $c$ be a cluster point of period $n$ of $F$. Then the combinatorial rotation number is defined as follows. The first return map, $F^{\circ n}$, maps the star of the cluster, $X_{F}$ to itself. Label the arms of the star (the $0$-internal rays) which belong to one of the critical orbits (it does not matter which) cyclically in anticlockwise order by $\ell_1$, $\ell_2, \ldots \ell_n$ (the initial choice of $\ell_1$ is not important). Then for each $k$, there exists $p$ such that $F^{\circ n}$ maps $\ell_{k}$ to $\ell_{k+p}$, subscripts taken modulo $n$. We then say the combinatorial rotation number is $\rho = \rho(F) = p/n$. 
\end{defn}

We can now define the combinatorial data of a rational map with a cluster cycle to be the pair $(\rho,\delta)$ where $\rho$ is the combinatorial rotation number and $\delta$ is the critical displacement. As shown in \cite{Thurstoneq}, the combinatorial data of a map is enough to classify it (in the sense of Thurston) in the fixed case for all degrees and the period two case if the map is quadratic.

\subsection{Matings}

We will be constructing rational maps using matings. The mating of polynomials was first mentioned by Douady in \cite{Douady:Bourbaki}. Informally, the construction allows us to take two complex polynomials $f$ and $g$ (along with their filled Julia sets $K(f)$ and $K(g)$) and paste them together to construct branched cover of the sphere. We will informally consider the mating operation to be a map from the ordered pairs $(f,g)$ of polynomials to the space of branched covers of the sphere. For further reading on this subject, see \cite{TanLei:1990,Milnor:mating,Wittner,TanShish}.

Let $f$ and $g$ be two monic degree $d$ polynomials defined on $\C$. In this paper, $f$ and $g$ will be unicritical but this is not needed in general. We define $\widetilde{\C} = \C \cup \{ \infty \cdot e^{2 \pi i t} : t \in \R / \Z \}$, the complex plane with the circle at infinity. We then extend the two polynomials to the circle at infinity by defining
\[
  f(\infty \cdot e^{2 \pi i t}) =  \infty \cdot e^{2 d \pi i t} \quad \text{and} \quad g(\infty \cdot e^{2 \pi i t}) =  \infty \cdot e^{2 d \pi i t}.
\]
Label this extended dynamical plane of $f$ by $\widetilde{\C}_f$ and the extended dynamical plane of $g$ by $\widetilde{\C}_g$. We then create a topological sphere by gluing the two extended planes together along the circle at infinity. More formally, we define
\[
 S^2_{f,g} = \widetilde{\C}_f \uplus \widetilde{\C}_g / \sim
\]
where $\sim$ is the relation which identifies the point $\infty \cdot e^{2 \pi i t} \in \widetilde{\C}_f$ with the point $\infty \cdot e^{- 2 \pi i t} \in \widetilde{\C}_g$. We now define a new map, the formal mating $f \uplus g$, on this new space $S^2_{f,g}$ by defining
\begin{align*}
	f \uplus g|_{\widetilde{\C}_f} \, =& \, f \quad \textrm{and} \\
	f \uplus g|_{\widetilde{\C}_g} \, =& \, g. 
\end{align*}

We now wish to define an alternative type of mating, called the \emph{topological mating}. First we require a brief discussion on the theory of external rays. Suppose the (filled) Julia set of the degree $d \geq 2$ monic polynomial $f \colon \CC \to \CC$ is connected. Recall, that by B\"{o}ttcher's theorem, there is a conformal isomorphism
\[
  \phi = \phi_f \colon \CC \setminus \overline{\D} \to \CC \setminus K(f)
\]
which can be chosen so that it conjugates $z \mapsto z^d$ on $\CC \setminus \overline{\D}$ with the map $f$ on $\CC \setminus K(f)$. 

\begin{defn}\label{d:extray}
 Consider the radial line $r_t = \{ r \exp{2 \pi i t} : r > 1 \} \subset \C \setminus \D$. Then we call the set
\[
  R_{f}(t) = \phi_f (r_t)
\]
 the external ray of angle $t$.
\end{defn}

If $K(f)$ (equivalently $J(f)$) is locally connected, we can define the Carath\'{e}odory semiconjugacy $\gamma \colon \R / \Z$ such that
\[
 \gamma(t) = \gamma_f(t) = \lim_{r \to 1} \phi_f(r \exp(2 \pi i t)).
\]
The point $\gamma_f(t) \in K(f)$ is called the landing point of the external ray $R_{f}(t)$. Before moving on, we introduce some terminology. Given a polynomial of the form $f_c(z) = z^d + c$ with locally connected Julia set, it is easy to see that the points $\beta_k = \gamma(k/(d-1))$ must be fixed. We call these the $\beta$-fixed point of $f_c$. There exists at most one other fixed point, which we will call the $\alpha$-fixed point of $f_c$. We can also construct parameter rays analogously in the parameter plane. Let $\Phi$ be the (unique) Riemann map  which gives a conformal isomorphism between $\CC \setminus \overline{\D}$ and $\CC \setminus \M_d$, with $\Phi (\infty) = \infty$ and so that $\Phi$ is asymptotic to the identity at infinity. Then the parameter ray of angle $\theta$ is the set
\[
  R_{\M_d}(t) = \Phi( \{ r\exp{2 \pi i t} : r > 1\}).
\]
The notion of landing of rays is defined similarly as to the case with external rays.

We now can define the topological mating of two monic degree $d$ polynomials $f_1$ and $f_2$, assuming they have locally connected Julia set. We first define the ray-equivalence relation $\sim$ on $S^2_{f,g}$. The equivalence relation $\sim_f$ on $\widehat{\C}_f$ is generated by $x \sim_f y$ if and only if $x,y \in \overline{R}_{f}(t)$ for some $t$. Notice that the closure of the external ray contains both the landing point and the point on the circle at infinity. Define a similar equivalence relation on $\widehat{\C}_g$. Then the equivalence relation $\sim$ will be generated $\sim_f$ on $\widehat{\C}_f$ and $\sim_g$ on $\widehat{\C}_g$. We denote the equivalence class of $x$ under this relation by $[x]$.

Denote the Carath\'{e}odory semiconjugacy derived from $f_{j}$ by $\gamma_{j}$, so that $\gamma_{j}$. We see that the ray equivalence relation restricts to an equivalence relation $\sim'$ on the disjoint union of $K(f_1)$ and $K(f_2)$ by
\[
    \gamma_{1}(t) \sim' \gamma_{2}(-t) \quad \textrm{ for each } \quad t \in \R / \Z.
\]
We define $K(f_{1}) \Perp K(f_{2})$ to be the quotient topological space $S^2_{f,g} / \sim$, where every equivalence class is
identified to a point. Making use of the fact that $\gamma_{j}(dt) = f_{j}(\gamma(t))$, we can piece together $f_{1} |_{K(f_{1})}$ and $f_{2} |_{K(f_{2})}$ to form a continuous map which we call $f_{1} \Perp f_{2}$. This is the topological mating. In nice cases, this quotient space $K_1 \Perp K_2$ is a topological sphere. We can think of the topological mating as the formal mating, where the external rays have been ``drawn tight'' using the ray equivalence relation. We remark that in this paper, we will only be considering the mating of (monic unicritical) hyperbolic polynomials of degree $d \geq 2$. By the results of \cite{Pilgrim:Jordan}, this means the Julia sets will be locally connected and so the topological mating will be well defined. Now suppose we have constructed the topological mating $f_1 \Perp f_2$. We say that a rational map $F$ is the geometric mating of the two monic polynomials $f_1$ and $f_2$ if there exists a topological conjugacy $h$ which is orientation preserving and holomorphic on $\stackrel{\circ}{K_1}$ and $\stackrel{\circ}{K_2}$ satisfying
\[
	h \circ (f_1 \Perp f_2) = F \circ h.
\]
In this case, we will write $F \cong f_1 \Perp f_2$. If there exist two (or more) rational maps satisfying the above, we say the mating is shared. 

\subsection{Thurston Equivalence}

We now discuss how the mating can be used to define a rational map on the Riemann sphere. As it stands the formal mating is a branched covering of the sphere. To provide a conformal structure to the topological sphere, we need to invoke the results of Thurston on the classification of branched covers. Recall a branched covering $F S^2 \to S^2$ is called postcritically finite if the postcritical set
\[
  P_F = \bigcup_{n > 0} F^{\circ n}(\{ z : z \text{ is a critical point of $F$} \}).
\]
is finite.

\begin{defn}
    Two postcritically finite orientation preserving branched self-coverings with labelled critical points $(F, P_{F})$ and $(G,P_{G})$ of $S^2$ are said to be \emph{Thurston equivalent} (alternatively, \emph{combinatorially equivalent} or just \emph{equivalent}) if there exists orientation preserving homeomorphisms $\phi_{1}, \phi_{2} \colon S^{2} \to S^{2}$ such that
    \begin{itemize}
     \item{$\phi_{1} |_{P_{F}} = \phi_{2} |_{P_{F}}$}
     \item{The following diagram commutes.
        \[
             \xymatrix{       (S^{2}, P_{F}) \ar[r]^{\phi_{1}} \ar[d]_{F}    & (S^{2}, P_{G}) \ar[d]^{G}
            \\
                (S^{2}, P_{F}) \ar[r]_{\phi_{2}}                       & (S^{2}, P_{G})}
        \]}
    \item{$\phi_{1}$ and $\phi_{2}$ are isotopic via homeomorphisms $\phi_{t}$, $t \in [0,1]$ satisfying $\phi_{0} |_{P_{F}} = \phi_{t} |_{P_{F}} = \phi_{1} |_{P_{F}}$ for each $t \in [0,1].$}
        \end{itemize}
When two branched coverings are equivalent in this way, we write $F \cong G$.
\end{defn}

 Let $\Gamma = \{ \gamma_{1}, \gamma_{2}, \ldots , \gamma_{n} \}$ be a collection of curves in $S^{2}$. If the $\gamma_{i} \in \Gamma$ are simple, closed, non-peripheral, disjoint and non-homotopic relative to $P_F$ then we say $\Gamma$ is a multicurve. We say the multicurve is $F$-stable if for any $\gamma_{i} \in \Gamma$, all the non-peripheral components of $F^{-1}(\gamma_{i})$ are homotopic rel $S^{2} \setminus P_{F}$ to elements of $\Gamma$. Given an $F$-stable multicurve, we can define a non-negative matrix $F_{\Gamma} = (f_{ij})_{n \times n}$ in the following natural way. For each $i,j$, let $\gamma_{i,j,\alpha}$ be the components (these are all simple, closed curves) of $F^{-1}(\gamma_{j})$ which are homotopic to $\gamma_{i}$ in $S^{2} \setminus P_{F}$. Now define
\[
	F_{\Gamma}(\gamma_{j}) = \sum_{i,\alpha} \frac{1}{\deg F |_{\gamma_{i,j,\alpha}} \colon \gamma_{i,j,\alpha} \to \gamma_{j}} \gamma_{i}
\] 
\noindent where $\deg$ denotes the degree of the map. By standard results on non-negative matrices (see \cite{Seneta}), this matrix $f_{ij}$ will have a leading non-negative eigenvalue $\lambda$. We write $\lambda(\Gamma)$ for the leading eigenvalue associated to the multicurve $\Gamma$.

\begin{defn}
 The multicurve $\Gamma$ is called a Thurston obstruction if $\lambda(\Gamma) \geq 1$.
\end{defn}

\begin{thm}[Thurston's Theorem]
 A post-critically finite branched covering $F \colon S^{2} \to S^{2}$ of degree $d \geq 2$ with hyperbolic orbifold is equivalent to a rational map $R$ on the Riemann sphere if and only if for any $F$-stable multicurve $\Gamma$ we have $\lambda(\Gamma,F)<1$. In that case the rational function $R$ is unique up to conjugation by an automorphism of the Riemann sphere $\CC$.
\end{thm}

The requirement of $F$ having a hyperbolic orbifold in the above theorem will not concern us in this paper, since all of our branched coverings will have hyperbolic orbifolds. Indeed, if a branched covering $F$ has non-hyperbolic orbifold, then it must have $P_F \leq 4$ and the pre-image of the post-critical set must be contained in the union of the postcritical set and the set of critical points of $F$. The interested reader should refer to \cite{DouadyHubbard:1993} for a proof and discussion of the above result. Though Thurston's theorem is a very strong result, it has some problems in applications, namely it is difficult to check in general for Thurston obstructions.

\begin{defn}
    A multicurve $\Gamma = \{ \gamma_{1}, \gamma_{2}, \ldots , \gamma_{n} \}$ is a Levy cycle if for each $i =1,\ldots,n$, the curve $\gamma_{i-1}$ (or $\gamma_{n}$ if $i = 1$) is homotopic to some component $\gamma_{i}'$ of $F^{-1}(\gamma_{i})$ (rel $P_{F}$) and the map $F \colon \gamma_{i}' \to \gamma_{i}$ is a homeomorphism.
\end{defn}

The following result, the culmination of work by Rees, Shishikura and Tan Lei, greatly simplifies the search for Thurston obstructions in the bicritical case.

\begin{prop}
    In the bicritical case, $F$ has a Levy cycle if and only if it has a Thurston obstruction.
\end{prop}

A related result in \cite{TanLei:1990} is the following. 

\begin{thm}\label{t:Tanequivalences}
	Let  $f_{1} , f_{2}$ be monic unicritical polynomials of the form $f_{i}(z) = z^d + c_i$ and with $\alpha$-fixed points labelled as $\alpha_{1}$ and $\alpha_{2}$ respectively. Then the following are equivalent.
		\begin{enumerate}
			\item{$F$ has a (good) Levy cycle $\Gamma = \{ \gamma_{1},\ldots,\gamma_{n} \}$.}
			\item{$F$ has a ray-equivalence class $\tau$ containing closed loops and two distinct fixed points.}
			\item{$[\alpha_{1}] = [\alpha_{2}]$}
			\item{$f_{1}$ and $f_{2}$ are in conjugate limbs of the degree $d$ multibrot set.}
		\end{enumerate}
\end{thm}

This result significantly reduces the work required to check whether the mating of two unicritical monic polynomials is obstructed, since it suffices by condition 3 to check that the ray classes of $[\alpha_1]$ and $[\alpha_2]$ are disjoint. We finish with a slight generalisation of Theorem~\ref{t:Tanequivalences}, which we will use in the next section.

\begin{lem}\label{TanShish}
	Let $F$ be a mating. Let $[x]$ be a periodic ray class such that $[x]$ contains a closed loop. Then each boundary curve of a tubular neighbourhood of $[x]$ generates a Levy cycle.
\end{lem}

By a result of Rees \cite{Rees:Param}, in the hyperbolic postcritically finite case, if the formal mating of $f_1$ and $f_2$ is Thurston equivalent to the rational map $F$, then $F$ will be the geometric mating of $f_1$ and $f_2$. Hence we can use each of the different notions of mating when discussing a mating, without worrying that a different choice will affect the Thurston class. 

\subsubsection{A combinatorial view of matings}

In this section we discuss the (periodic) ray classes that occur in the formal matings, which are then collapsed in the topological mating. This discussion is extremely natural and will allow us to focus on the important ray classes (those which become the cluster points) in the later sections.

\begin{lem}\label{disjointness}
Let $F = f_1 \uplus f_2$ be the formal mating of two hyperbolic polynomials which has no Thurston obstruction.	Let $z_0,F(z_0) = z_1,\ldots,F^{\circ(n-1)} = z_{n-1}$ be a period $n$ orbit of $f_1$ which is contained in $J(f_1)$ and has combinatorial rotation number different from $0$. Then the periodic ray classes $[z_0],[z_1],\ldots,[z_{n-1}]$ are pairwise disjoint.
\end{lem}

\begin{proof}
Suppose there exists $k$ with $[z_0] = [z_k]$. Then there exists a path through external rays $\gamma$ from $z_0$ to $z_k$. The map $F^{\circ n}$ will take $z_0$ to $z_0$ and $z_k$ to $z_k$ and takes the path $\gamma$ to some path $\gamma'$, which is a path from $z_0$ to $z_k$. But $\gamma'$ is not equal to $\gamma$, since the first return map to $z_0$ and $z_k$ will permute the external rays landing there. Hence the union $\gamma \cup \gamma'$ contains a loop; and so the mating will be obstructed. This contradiction completes the proof.
\end{proof}

\begin{lem}\label{l:matingrayclass}
	If the mating of two hyperbolic polynomials is not obstructed, each periodic ray class contains at most one periodic branch point with non-zero combinatorial rotation number.
\end{lem}

\begin{proof}
Let $w_{0}$ and $z_{0}$ be two periodic branch points with non-zero combinatorial rotation number, such that $[w_{0}] = [z_{0}]$. We will show that the periods of $z_0$ and $w_0$ are equal. Let the period of $z_0$ be $n$. Then the map $F^{\circ n}$ maps $z_0$ to itself, the periodic ray class $[z_0]$ to itself and $w_0$ to $w_n = F^{\circ n}(w_0)$. Then we must have $[w_0] = [z_0] = [w_n]$ and so $w_0 = w_n$ by Lemma~\ref{disjointness}. Hence the period of $w_0$ is divisible by $n$. An analogous argument shows the period of $w_0$ divides the period of $z_0$, and so the periods are the same. Denote this common period by $n$.  

Now let $\gamma$ be a path through external rays from $w_{0}$ to $z_{0}$. Since none of the rays meet a pre-critical point (since the maps $f_1$ and $f_2$ are hyperbolic), the $n$th iterate of $\gamma$, which we call $\gamma'$, will also be a path from $w_{0}$ to $z_{0}$. Since the external rays at $w_{0}$ and $z_{0}$ are permuted under the first return map, $\gamma \neq \gamma'$, and so the curve $\gamma \cup \gamma'$ contains a loop, and so the mating is obstructed. \end{proof}

The importance of this second lemma is it will tell us the nature of the maps which make up the matings that produce maps with cluster cycles. Intuitively, it is clear that, to create a rational map with a cluster cycle of period $p$ and rotation number $\rho$, it is necessary for one of the maps to have a periodic cycle of period $p$ and combinatorial rotation number $\rho$, so that this cycle for the polynomial will become the cluster point cycle. We will show that this intuition is correct, at least in the cases for periods one and two (though analogous proofs exist for all periods). The above lemma, on the other hand, sheds some light on the behaviour of the other polynomial in the mating. In short, it will allow us to say that the second polynomial must belong to a satellite component in the Multibrot set. 

Note that if we have a periodic ray class which contains a point $p$ with combinatorial rotation number not equal to $0$, then the ray class must contain a sort of rotational symmetry: each global arm at the point $p$\footnote{a global arm is a component of the complement to $p$ in the ray class} is homeomorphic to each of the other global arms. Furthermore, if the ray class is periodic, then it cannot contain a strictly pre-periodic point, and all points in the ray class have period dividing the period of the ray class. 

\begin{lem}\label{l:armslemma}
	Suppose a periodic ray class contains a point $p$ which has combinatorial rotation number different from $0$. Then given any periodic orbit $\mathcal{O} = \mathcal{O}(z_0)=\{z_0,z_1,\ldots,z_{n-1} \}$, the intersection between any global arm at $p$ and $\mathcal{O}$ contains at most one point. 
\end{lem}

\begin{proof}
Clearly if $[p] \neq [z_i]$ for any $i$ then the intersection of $\mathcal{O}(z)$ with any global arm will be empty. So assume without loss of generality that $[p] \cap [z_0] \neq \varnothing$. Let $\ell$ be a global arm at $p$ in the periodic ray class. There exists some $k$ so that $F^{\circ k}(\ell) = \ell$. If $\ell \cap \mathcal{O}$ is empty, then $\ell$ also contains no pre-images of points in $\mathcal{O}$, since otherwise some forward image of $\ell$ will contain a point in $\mathcal{O}$, and so $\ell \cap \mathcal{O}$ would be non-empty, a contradiction. Since $\ell$ contains no pre-images of points in $\mathcal{O}$, no forward iterate of $\ell$ can contain points of $\mathcal{O}$. Since global arms are mapped homeomorphically onto global arms, this means all the global arms have empty intersection with the orbit $\mathcal{O}$. Hence every global arm at $p$ must contain at least one point of $\mathcal{O}$.
	
Suppose the global arm $\ell$ contains $r>1$ elements of $\mathcal{O}$, $z_0,\ldots,z_{r-1}$. Under the first return map to $\ell$, these elements must be permuted, since the orbit is periodic. Without loss of generality, we can assume the first element of $\ell \cap \mathcal{O}$ on the global arm (in terms of distance from $p$) is $z_0$. Then the second point (again, in terms of distance in the tree from $p$) in $\ell \cap \mathcal{O}$ is $F^{\circ k}(z_0) = z_k$ for some $k$. Let $\gamma$ be the sub-arm from $p$ to $z_0$. Then $\gamma' = F^{\circ k}(\gamma)$ is contained in $\ell$ and will be a path from $p$ to $z_k$, since the global arm $\ell$ maps homeomorphically onto itself under the return maps (since the map $F$ is a homeomorphism on the graphs of the ray equivalence classes). In particular, $z_0 \in \gamma'$. However, this means that there has to be a pre-image of $z_0$ in the path $\gamma$. However, by construction, there are no points in the orbit of $z_0$ in the interior of $\gamma$ and there also does not exist any preperiodic points in $\gamma$. This contradiction means that the global arm $\ell$ must contain at most one point in $\mathcal{O}$. \end{proof}

\subsection{Main results}

The focus of this paper will be the proof of the following two theorems.

\begin{mthm}\label{mainthm1}
 Suppose that $F$ is a bicritical rational map with a fixed cluster point and the combinatorial rotation number is $p/q$. Then the following hold.
   \begin{itemize}
      \item{$F$ is the mating of a $q$-rabbit (whose $\alpha$-fixed point has combinatorial rotation number $p/q$) and another map $h$.}
      \item{The map $h$ has an associated angle with angular rotation number $(q-p)/q$.}
   \end{itemize}
\end{mthm}

The theorem in the period two cluster case is only concerned with the quadratic case.

\begin{mthm}\label{mainthm2}
  Suppose that $F$ is a quadratic rational map with a period two cluster point and the combinatorial rotation number is $p/q$. Then the following hold. 
  \begin{itemize}
    \item{$F$ is a mating.}
    \item{Precisely one of the maps in the mating belongs to the $1/2$-limb of the Mandelbrot set and this map is one of the two period $2q$ maps which belong to the $p/q$-sublimb of the period two component of the Mandelbrot set.}
    \item{The other map has an associated angle with 2-angular rotation number $(q-p)/q$.}
  \end{itemize}
\end{mthm}

We also will prove the following fact about the multiplicity of the shared matings in the period two cluster case.

\begin{mthm}\label{t:shares}
  In the period two cluster cycle case, each combinatorial data can be obtained in either two, three or four ways and all of these multiplicities are obtained.
\end{mthm}

\section{Fixed cluster points}\label{fixedclust}

Our goal in this section will be to prove Theorem~\ref{mainthm1}. The results of this section are relatively elementary, but provide an interesting insight into the the differences between the fixed case and the period two case. It will also allow us to introduxw some terminology that will be of use in the period two cluster cycle case.

\begin{prop}\label{p:adamproof}
    Suppose $F$ is a bicritical rational map with a fixed cluster point, where both critical points have period $n$. Then the two critical point components cannot be adjacent. In other words, the critical displacement cannot be $1$ or $2n-1$. 
\end{prop}

\begin{proof}
      We will assume, to obtain a contradiction, that $U_{0}$ is the adjacent component which is anticlockwise of $V_{0}$ in the cyclic ordering of components around the cluster point $c$, and that $U_0$ and $V_0$ are the components containing the critical points of $F$. We will use the notation that $U_{k} = F^{\circ k}(U_{0})$ and $V_{k} = F^{\circ k}(V_{k})$ to label the critical orbit components. Let $\gamma_1$ be a curve which loops around $v_1$ and $v_2$ as in Figure~\ref{f:Adamproof2} so that the only critical orbit Fatou components intersecting with $\gamma_1$ are $U_1$ and $U_2$. Now define $\gamma_n = F^{\circ (n-1)}(\gamma_1)$ for $n = 1,\ldots,n-1$.  Then $\Gamma = \{ \gamma_1, \ldots, \gamma_n \}$ is a Levy cycle. \end{proof}


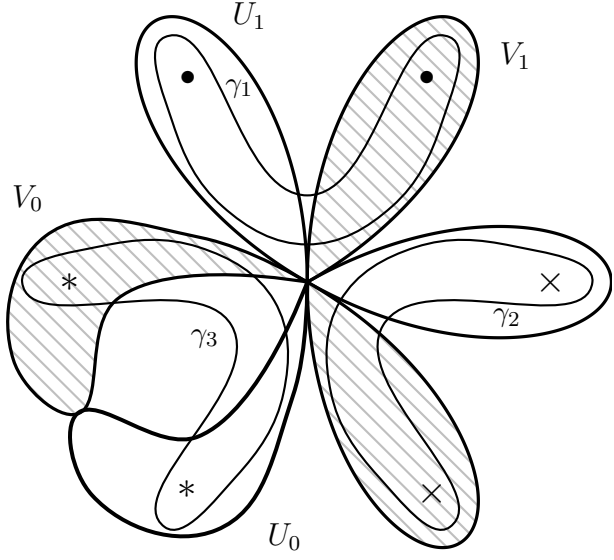
\begin{figure}[ht]
\begin{center}
\begin{pspicture*}(-4,-3.75)(4,3.75)
    \psplot[plotstyle=curve,plotpoints=500,polarplot=true,linewidth=1.2pt]%
        {90}{150}{4 3 x mul 90 add sin mul}
    \psplot[plotstyle=curve,plotpoints=500,polarplot=true,linewidth=1.2pt]%
        {150}{210}{4 3 x mul 90 add sin mul}
    \psplot[plotstyle=curve,plotpoints=500,polarplot=true,fillstyle=vlines,hatchcolor=lightgray,linewidth=1.2pt]%
        {30}{90}{4 3 x mul 270 add sin mul}
    \psplot[plotstyle=curve,plotpoints=500,polarplot=true,fillstyle=vlines,hatchcolor=lightgray,linewidth=1.2pt]%
        {90}{150}{4 3 x mul 270 add sin mul}
        \SpecialCoor
            \psccurve[curvature=1 0.25 0](0.5;90)(1.5;138)(2.75;130)(3.75;120)
                (1.15;90)(3.75;60)(2.75;50)(1.5;42)(0.5;90)
            \pscurve[curvature=0.9 1 0,fillstyle=vlines,hatchcolor=lightgray,linewidth=1.2pt]
                (0,0)(1.5;157)(3.5;170)(4;195)(3.5;210)(2.5;185)(0,0)
            \pscurve[curvature=0.9 1 0,fillcolor=white,linewidth=1.5pt]
                (0,0)(1.5;263)(3.5;250)(4;225)(3.5;210)(2.5;235)(0,0)
            \psccurve[curvature=1 0.25 0](0.5;330)(1.5;18)(2.75;10)(3.75;0)
                (1.15;330)(3.75;300)(2.75;290)(1.5;282)(0.5;330)
            \psccurve[curvature=1 0.25 0](0.5;210)(1.5;258)(2.75;250)(3.75;240)
                (1.15;210)(3.75;180)(2.75;170)(1.5;162)(0.5;210)
            \uput{16pt}[340](3.25;120){\large$\gamma_{1}$}
            \uput{16pt}[220](3.25;0){\large$\gamma_{2}$}
            \uput{5pt}[210](1.15;210){\large$\gamma_{3}$}
        \NormalCoor
    \uput{3}[60](0,0){\Large$\bullet$}
    \uput{3}[120](0,0){\Large$\bullet$}
    \uput{3}[180](0,0){\LARGE$\ast$}
    \uput{3}[240](0,0){\LARGE$\ast$}
    \uput{3}[0](0,0){\LARGE$\times$}
    \uput{3}[300](0,0){\LARGE$\times$}
    \uput{3.8}[48](0,0){\Large$V_1$}
    \uput{3.4}[102](0,0){\Large$U_1$}
    \uput{3.6}[163](0,0){\Large$V_0$}
    \uput{3.2}[265](0,0){\Large$U_0$}
\end{pspicture*}
\end{center}
\caption{The multicurve $\Gamma$.}
\label{f:Adamproof2}
\end{figure}

This result tells us that the two critical points cannot be in adjacent Fatou components in the cluster. We will actually show that these are the only cases that cannot be obtained as combinatorial data for rational maps. We now prove the first part of our classification: that one of the maps must be an $q$-rabbit. Here, a $q$-rabbit will be any map which belongs to a hyperbolic component which bifurcates off of the (unique) period 1 component in the degree $d$ multibrot set $\M_d$.

\begin{prop}\label{p:nrabbit}
	Let $f_1$ and $f_2$ be monic unicritical polynomials with a period $q$ superattracting orbit. Suppose $F \cong f_{1} \Perp f_2$ is a rational map that has a fixed cluster point. Then one of $f_{1}$ or $f_{2}$ is an $q$-rabbit.
\end{prop}

\begin{proof}
  A degree $d$ rational map has $d+1$ fixed points (counting multiplicity), and so we see that one of the $\alpha$-fixed points, $\alpha_1$ or $\alpha_2$, must become the fixed cluster point in the rational map, since otherwise we would have $[\alpha_1] = [\alpha_2]$ and the mating would be obstructed, by Theorem~\ref{t:Tanequivalences}. Without loss of generality, let the cluster point of $F$ be $[\alpha_1]$. 

  $\alpha_1$ is the landing point of $k$ external rays and these are permuted cyclically under iteration by the map $f_1$. Since the mating is not obstructed, $[\alpha_1]$ is a tree and the global arms of the tree are permuted cyclically and homeomorphically under $f_1 \uplus f_2$. Furthermore, because the map $f_1 \uplus f_2$ is a homeomorphism on this tree, each global arm contains at most one root point of a critical orbit Fatou component of $f_2$. Since there must be $q$ root points of critical orbit Fatou components in this ray equivalence class, there are $n$ global arms at $\alpha_1$ and so $q$ external rays landing at $\alpha_1$. But then $f_1$ has a period $n$ superattracting cycle and $n$ external rays landing on its $\alpha$-fixed point, so it is an $q$-rabbit. \end{proof}

Indeed, we can say slightly more. If the rational map has a fixed cluster point with combinatorial rotation number $\rho = p/q$, then the rabbit is actually the $p/q$-rabbit, the rabbit whose $\alpha$-fixed point with rotation number $p/q$. This is easy to see by noticing that the combinatorial rotation number at the $\alpha$-fixed point is given by the ordering of the permutation of the external rays, and it is precisely these external rays which form the arms of the periodic ray class which become the cluster point.

\subsection{Properties of the non-rabbit}\label{ss:1non-rab}

To complete the classification of the maps which mate to give a rational function with a fixed cluster point, we need to study the properties of the map which partners the $q$-rabbit in the mating. Hence we know turn our attention to the polynomials which partner the rabbits in the matings. Indeed, our classification of this complementary map requires us to take into account the ordering of the external rays which land at the $\alpha$-fixed point of the rabbit. This classification will follow the considerations found in \cite{Bullett-Sentenac} and \cite{Blokhetal}. In particular, we need to provide the definition of the rotation number of an angle.

\begin{defn}\label{d:1-ARN}
	Let $\theta \in \mathbb{S}^{1}$ be periodic of period $q$ under the map $\sigma_d \colon t \mapsto dt$, so that $\theta = a/(d^{q}-1)$ for some $a$. Label the angles $d^j\theta$, $j=1,\ldots,q$ cyclically by $\theta_1,\theta_2,\ldots,\theta_n$ with $\theta_1 = \theta$. Then we say that the angle $\theta$ has (angular) rotation number $p/q$ if
\[
	d \theta_k = \theta_{k+p \mod q}
\]
 for each $k$.
\end{defn}

Notice that if $\theta$ has rotation number $p/q$, then the angles $d \theta$, $d^2 \theta, \ldots$ have rotation number $p/q$ also. Hence we can equally well talk of the orbit of angles having angular rotation number $p/q$.

\begin{defn}
 Let $f$ be a hyperbolic polynomial belonging to the degree $d$ multibrot set $\M_d$. We say the angle $\theta$ is associated with the polynomial $f$ if the external ray of angle $\theta$ lands at a (not necessarily principle) root point of the critical value component Fatou component of $f$.
\end{defn}

The following is well-known and gives an extremely useful link between the dynamical plane and the parameter plane; see for example \cite{Milnor:mandel}.

\begin{prop}
  The angle $\theta$ is associated to $f$ if and only if the parameter ray of angle $\theta$ lands at a (not necessarily principle) root point of the hyperbolic component containing the map $f$.
\end{prop}

In the sequel, we use the notation $A_k$ to denote an arc of the form $(k/(d-1),(k+1)/(d-1)) \subset S^1$.

\begin{lem}\label{hangles}
 Suppose $F = f \Perp h$ is a rational map with a fixed cluster point and $f_1$ is the $p/q$-rabbit. Then one of the angles associated with $h$ has rotation number $(q-p)/q$. Moreover, in degree $d$, all the angles in the forward orbit of this angle lie in an arc $A_j$.
\end{lem}

\begin{proof}
  All rays landing on the $\alpha$-fixed point of $f$ have angles $\theta_i$ (in cyclic order) with rotation number $p/q$ and lie in arc $A_k$ for some $k \in \{ 0,1,\ldots,d-2 \}$. Since the only periodic biaccessible point on $J(f_1)$ is $\alpha_1$, (\cite{Wittner}, Claim 10.1.1) the root points of the critical orbit Fatou components must be the landing points of the rays of angles $-\theta_i$. Each of these angles has rotation number $(q-p)/q$, and precisely one of them is the angle of the ray landing at a root point of the critical value Fatou component of $h$. Since it lands at a root point of the critcal value component of $h$, it is one of the angles associated with $h$. The second statement simply follows from the fact that all the angles $\theta$ of the external rays landing at the $\alpha$-fixed point of an $n$-rabbit all lie in some arc $A_k$, and so the angles $-\theta$ must also lie in the arc $A_{d-k-2}$. 
\end{proof}

\begin{prop}\label{p:allobtained}
  All (admissible) combinatorial data can be obtained by matings in (precisely) $2(d-1)$ ways.
\end{prop}

\begin{proof}
  First we will construct a rational map $F \cong f \Perp h$ with combinatorial data $(\rho, \delta)$. $\rho$ is of the form $p/q$, so we take $f$ to be the $p/n$-rabbit whose corresponding angles lie in the arc $A_0$. We now make a judicious choice for our complementary map $h$. Label the angles of the rays landing at the $\alpha$-fixed point of $f$ in cyclic order by $\theta_1,\ldots,\theta_q$, starting with $\theta_1$ as the angle of the ray which lands anticlockwise from the critical value component of $f$. To get a critical displacement of $\delta = 2k-1$, we choose $h$ to be the rational map corresponding to the angle $-\theta_k$. The map $h$ does not lie in the conjugate limb to $f$ in $\M_d$ and so the mating exists and the rational map has combinatorial data $(\rho, \delta)$.

  To see that all combinatorial data is obtained with multiplicity $2(d-1)$, observe that the $p/q$-rabbit $f$ that was chosen in the previous paragraph could equally well have been the one whose corresponding angles lay in the arc $A_k$, with $k \in \{0,\ldots,d-2\}$. Then a similar argument allows us to find a complementary map $h$ which means $f \Perp h$ has combinatorial data $(p/q,\delta)$. Moreover, we can carry out the mating $h' \Perp f$, where $h'$ is the map for which $f \Perp h'$ has critical displacement $-\delta$ and again this holds for any choice of the $p/q$-rabbit $f$. This gives us the $2(d-1)$ distinct ways of forming the mating.

  To see that $2(d-1)$ is sharp, observe that by results of \cite{Blokhetal} (in particular, Theorem 2.8) there is precisely one orbit of angles in each arc $A_k$ which has rotation number $(q-p)/q$\footnote{In degree 2, this essentially says the well-known fact that there is a unique $p/q$-rabbit.}. These angles are precisely the angles $-\theta_1,\ldots,-\theta_q$ above. Two of these (namely $-\theta_1$ and $-\theta_n$) correspond to the $(q-p)/q$-rabbit and each of the others construct a rational map with a distinct critical displacement when mated with the appropriate $p/q$-rabbit. So no other maps can be mated with a $p/q$-rabbit to create a rational map with a fixed cluster point.
\end{proof}

\begin{rem}
  It follows from the above proof and by a simple counting argument that the maps $h$ such that $f \Perp h$ has a fixed cluster cycle (for some choice of the $p/q$-rabbit $f$) are precisley those maps $h$ who have a corresponding angle $\theta$ which has angular rotation number $p/n$ and whose orbit $\{ \theta, d\theta, \ldots , d^{q-1}\theta \}$ all lie in an arc $A_k$. 
\end{rem}

\subsection{Proof of Theorem~\ref{mainthm1}}

We now prove the first of our two main theorems. First we state a result from \cite{Thurstoneq}.

\begin{thm}\label{Fixedcase}[Theorem~A, \cite{Thurstoneq}]
  Suppose $F$ and $G$ are bicritical rational maps (with labelled critical points) with fixed cluster cycles with the same combinatorial data. Then $F$ and $G$ are equivalent in the sense of Thurston.
\end{thm}

\begin{proof}[Proof of Theorem~\ref{mainthm1}]
 By Proposition~\ref{p:nrabbit}, one of the maps must be an $n$-rabbit, $f$, and since combinatorial rotation numbers are preserved by matings and the fact that the $\alpha$-fixed point of $f$ the cluster point, the combinatorial rotation number of the $\alpha$ fixed point of $f$ must be $p/n$ as well. By Lemma~\ref{hangles}, one of the angles associated to $h$ must have rotation number $(n-p)/n$. Since all admissible combinatorial data can be obtained by matings (Proposition~\ref{p:allobtained}), then by Theorem~\ref{Fixedcase}, all rational maps with fixed cluster points are matings.
\end{proof}


\section{Period two cluster cycles}

The simplicity of the results in the fixed cluster point case perhaps could lead the reader to believe that the analogous results hold in the period two case. However, there is actually an increased level of complexity in the period two case, as we will discover. We will only consider the quadratic case in this section: a brief discussion of the higher degree case (and the difficulties in tackling it) is given after Theorem~\ref{t:preciselyone}. We know from the previous section that if we wanted to construct a rational map with a fixed cluster point that has combinatorial roation number $\rho$, then we need one of the maps to be the $\rho$-rabbit. We do not get such an exact statement. However, we can differentiate between the two polynomials: one of them must belong to the $1/2$-limb.

\begin{thm}\label{t:preciselyone}
Let $F \cong f_{1} \Perp f_{2}$ be a degree 2 rational map with a period 2 cycle of cluster points. Then precisely one of the maps $f_{1}$ or $f_{2}$ belongs to $\M_{1/2}$, the $1/2$-limb of $\M$.
\end{thm}

\begin{proof}
It is clear that both $f_{1}$ and $f_{2}$ cannot belong to the $(1/3,2/3)$-limb, since then the mating would be obstructed by Theorem~\ref{t:Tanequivalences}. Hence it only remains to show that both of $f_{1}$ and $f_{2}$ cannot lie outside $\M_{1/2}$. So assume $f_{1}$ and $f_{2}$ lie outside $\M_{1/2}$. If $f_{i}$ is not in $\M_{1/2}$, the external rays of angles $1/3$ and $2/3$ must land at distinct points. Since these angles have period 2 under angle doubling, they must land at points with period dividing 2. As $R_{1/3}$ and $R_{2/3}$ land at different points, these landing points must be a period 2 cycle.

Now notice that, under mating, we have the identifications
\[
	\gamma_{f_1}\left( 1/3 \right) \sim \gamma_{f_2} \left( 2/3 \right)  \quad \text{ and } \quad \gamma_{f_2} \left( 1/3 \right) \sim \gamma_{f_1} \left( 2/3 \right)
\]
and these points are not identified with each other, or any other points. In particular, these points cannot be cluster points. Since $f_{i}(\gamma_{f_i}(1/3)) = \gamma_{f_i}(2/3)$ and $f_{i}(\gamma_{f_i}(2/3)) = \gamma_{f_i}(1/3)$, these pairs form a period 2 cycle for the map $F \cong f_{1} \Perp f_{2}$. However, since $F$ already has a period 2 cycle (the cluster point cycle) by assumption, we see that this second period 2 cycle cannot exist, since a degree 2 rational map can only have one period 2 cycle. Hence both of the maps cannot lie outside $\M_{1/2}$, and so precisely one of them lies in $\M_{1/2}$.  \end{proof}

We notice that, similarly to the last section, we can separate the classification into an investigation of the map in $\M_{1/2}$ and, afterwards, we can study the map that does not belong to $\M_{1/2}$. Indeed as we will see, the period two cluster case is far more complicated than the relatively simple fixed cluster case. A further comment is required on the restriction to degree 2. On the one hand, this restriction is motivated by the fact that Thurston equivalence is only possible with the given combinatorial data in the quadratic case. However, it turns out that even a generalisation of Theorem~\ref{t:preciselyone} does not hold in the higher degree case. In particular, though it remains true that one of the maps must belong to \emph{a} $1/2$-limb\footnote{The notion of the $1/2$-limb is no longer unique in degrees greater than $2$} in the degree $d$ multibrot set $\M_d$, it is not true that precisely one of them has this property. Examples showing this is no longer the case can be found in \cite{Thurstoneq} (Section 4) and also in the forthcoming paper \cite{2SupAtt}. It is hoped a more detailed study of the higher degree case will be the subject of future work.

\subsection{Properties of the map in $\M_{1/2}$}

For each $\rho \in \Q / \Z$, there are in fact two maps in the $1/2$ limb that have a period two orbit with combinatorial rotation number $\rho$. It turns out that either of these can be used to create a rational map with a period two cluster cycle using matings.

\begin{prop}\label{p:2clusttunedrabbit}
	Let $f_1$ and $f_2$ be quadratic polynomials with period $2q$ superattracting orbits. Suppose $F \cong f_{1} \Perp f_2$ is a rational map that has a period 2 cluster cycle. Then one of $f_{1}$ or $f_{2}$ is either the tuning of the basilica by an $q$-rabbit (a``double rabbit''), or the (unique) other period $2q$ component lying in the wake of this double rabbit.
\end{prop}

\begin{proof}
  This proof is equivalent to showing that the map in $\M_{1/2}$ must belong to a $p/q$-sublimb of the period 2 component of $\M$. Notice that maps belonging to this sublimb are precisely those which have a period 2 point with combinatorial rotation number $p/q$. If $f \in \M_{1/2}$ then it has a period two cycle $\{ p_1 , p_2 \}$. Under mating, the equivalence classes of these points must also have period two, since otherwise we would have $[p_1] = [p_2]$, contradiciting Lemma~\ref{l:matingrayclass}. Since a quadratic rational map has at most one period two cycle, the classes $[p_1]$ and $[p_2]$ must become the cluster cycle. It now follows from Lemma~\ref{l:armslemma} that the combinatorial rotation number of this period two orbit will be $p/q$ for some $p$, and the only polynomials which have a period $2q$ superattracting cycle and a period two point with combinatorial rotation number $p/q$ are those two given in the statment of the proposition.
\end{proof}

An example of the position of the two maps in paramter space is given in Figure~\ref{f:twomaps}. It again follows easily that the combinatorial rotation number of the period two cycle is the same as the combinatorial rotation number of the period two cluster cycle in the resulting rational map. Figure~\ref{f:twomaps} shows the position of these two maps in the period 8, rotation number 1/4 case.
\begin{figure}[ht]
    \begin{center}
    \includegraphics[width=0.95\textwidth]{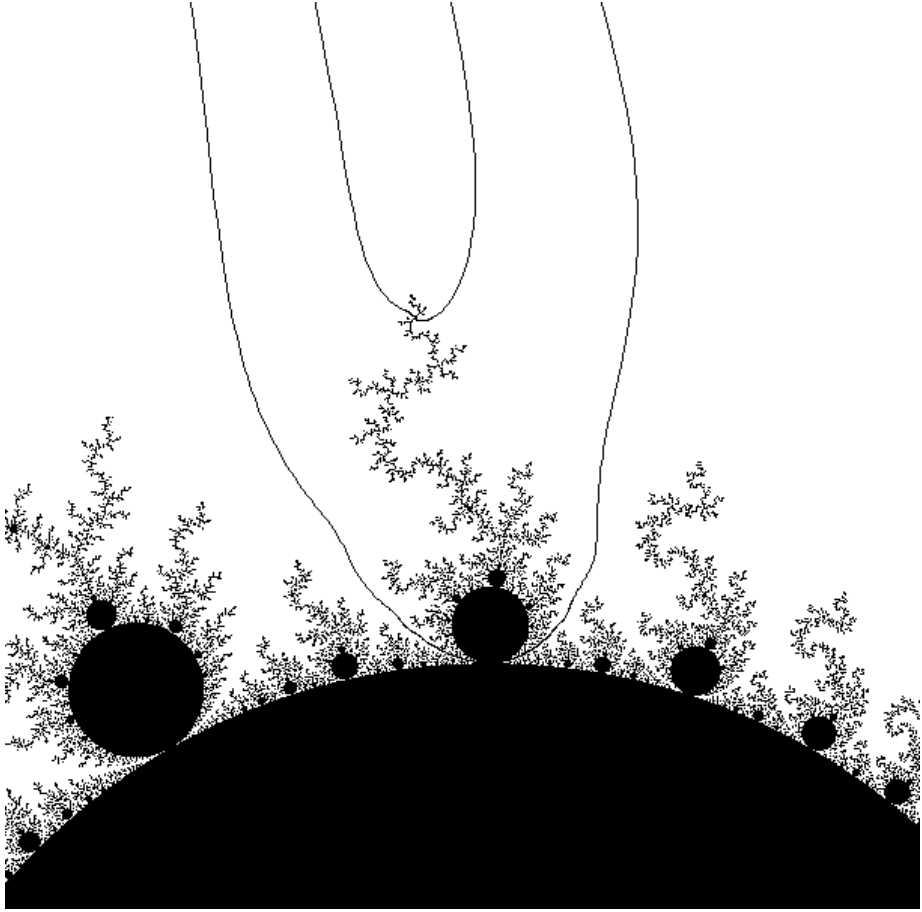}
    \end{center}
    \caption{The double rabbit component and secondary map component of period 8, rotation number 1/4 case.}
\label{f:twomaps}
\end{figure}
We remark for the moment that this classification is the best possible: it is not true in the period two cluster case that one of the maps must be the double rabbit, for example, since there are matings with the secondary maps which create period two cluster cycles. The matings with the double rabbit are the canonical examples, in that they behave similarly to the examples as found in Section~\ref{fixedclust}. We will consider the matings with the secondary map in Section~\ref{secondmap}. We denote by $f_{p/q}$ (respectively $g_{p/q}$) the double rabbit (respectively secondary map) with a period two orbit with combinatorial rotation number $p/q$.

\subsection{Properties of the complementary map}

We now attempt to prove some analogous results to those in Section~\ref{ss:1non-rab}. We start with a well-known lemma whose proof, which makes use of results from \cite{Milnor:mandel}, is omitted.

\begin{lem}\label{l:Wittner2}
Suppose that $z$ is a biaccessible periodic point in $J(f_{p/q})$. Then $z$ is either the $\alpha$-fixed point or belongs to the period 2 orbit $\{p_1,p_2\}$.
\end{lem}

As with the previous section, our description of the complementary map (the map which does not belong to the $1/2$-limb of the Mandelbrot set) will take place as a discussion of the associated angles. In fact, there are perhaps more descriptive ways of describing these maps (see \cite{Mythesis} for a discussion of the combinatorial classification using internal addresses), but the description with associated angles is the more complete for the moment. We first require an analogue to Definition~\ref{d:1-ARN}.

\begin{defn}\label{d:2-ARN}
  Let $\theta \in S^1$ be periodic of period $2q$ under the map $\sigma_2 : t \mapsto 2t$ (so $\theta = a/(2^{2q}-1)$ for some $a$). Consider the disjoint sets
\begin{align*}
    &A_{0} = A_{0}(\theta) = \left\{ \theta, 2^{2}\theta, 2^{4}\theta, \ldots, 2^{2k}\theta, \ldots , 2^{2(q-1)}\theta \right\} \quad\text{and}\\
    &A_{1} = A_1(\theta) = \left\{ 2 \theta, 2^{3} \theta, \ldots, 2^{2k+1}\theta, \ldots, 2^{2q-1}\theta \right\}.
\end{align*}
We say $\theta$ has (admissible) ($2$-)angular rotation number $p/q$ if the following conditions are satisfied.
\begin{itemize}
	\item{$A_0$ and $A_1$ lie in disjoint arcs of $S^1$}
	\item{The sets $A_0$ and $A_1$ have (angular) rotation number (in the sense of Definition~\ref{d:1-ARN}) $p/q$ under the map $\theta \mapsto 4 \theta$.}
\end{itemize}
\end{defn}

Clearly this is not a full generalisation of Definition~\ref{d:1-ARN}, since the first condition above introduces a restriction which was not used previously. It actually turns out that this restriction will be important for us. There are actually orbits of angles which satisfy the second condition but not the first (see \cite{Blokhetal}), but we do not want to consider such orbits in this paper and so the restriction allows us to ignore them. For the moment we will be assuming that the map in $\M_{1/2}$ is the double rabbit. As will be shown later, the set of complementary maps to the secondary map are a subset of the complementary maps to the double rabbit. 

\begin{prop}\label{p:hangles}
  Suppose $F \cong f_{p/q} \Perp h$ is a rational map with a period two cluster cycle. Then one of the angles associated to $h$ has $2$-angular rotation number $(q-p)/q$.
\end{prop}

\begin{proof}
  This is similar to Lemma~\ref{hangles}. The the ray classes of the period two orbit $\{ p_0, p_1 \}$ of $f_{p/q}$ will become the cluster cycle and we note that the angle of any external ray landing on one of the the $p_i$ has 2-angular rotation number $p/q$. Let $p_0$ be the periodic point whose external rays all lie in $(2/3,1/3)$ and label these angles in anticlockwise order (starting anywhere) by $\theta_1,\ldots,\theta_n$. We claim that $R^h_{-\theta_i}$ will land at the root point of a critical orbit Fatou component of $h$. If not, the ray class would have to contain another periodic biaccessible point of $J(f_{p/q})$ and by Lemma~\ref{l:Wittner2} the only other biaccessible point is $\alpha_f$ and by Lemma~\ref{l:matingrayclass} $[\alpha_f] \neq [p_1] \neq [p_2] \neq [\alpha_f]$. By Lemma~\ref{p:Marysresult}, one of the rays $R^h_{-\theta_i}$ will land at the root point of the critical value Fatou component of $h$, and so $-\theta_i$ will be one of the angles associated with $h$. Finally, as the angles in the orbit of $\theta_i$ have 2-angular rotation number $p/q$, the angles in the orbit  of $-\theta_i$ will have 2-angular rotation number $(q-p)/q$. 
\end{proof}

\begin{prop}\label{p:alldata2}
  All combinatorial data can be obtained.
\end{prop}

\begin{proof}
  We will show that any combinatorial data $(\rho,\delta) = (p/q,2k+1)$ can be obtained by a mating of the form $f_{p/q} \Perp h$. Clearly, if $h$ is chosen appropriately, then the rational map formed by this mating will have a period two cluster cycle with combinatorial rotation number $p/q$, since matings preserve the combinatorial rotation number. Label the angles of the external rays landing at $p_0 \in J(f_{p/q}))$ by $\theta_0,\ldots,\theta_{q-1}$, where $\theta_0$ is the angle of the external ray which approaches $p_0$ immediately anticlockwise from the critical point component of $f_{p/q}$. Then to get $\delta = 2k+1$, let $h$ be the map associated with the parameter ray of angle $-\theta_k$. Since $\theta_k \in (2/3,1/3)$, then $-\theta_k \in (2/3,1/3)$ also, so the mating is not obstructed by Tan Lei's theorem. Also, since $-\theta_k$ is associated to $h$, the external ray of angle $-\theta_k$ lands at the base point of the critical value component of $h$. Thus the rational map $F \cong f_{p/q} \Perp h$ exists and has combinatorial data $(\rho,\delta)$.
\end{proof}

A converse to the Proposition~\ref{p:hangles} also exists, if one places a suitable restriction on the angles that can be associated with the complementary map $h$.

\begin{prop}\label{p:hanglesconverse}
  Suppose $h$ has an associated angle $\theta$ which has $2$-angular rotation number $(q-p)/q$ and that this angle lies in $(2/3,1/3)$. Then $F \cong f_{p/q} \Perp h$ (and $F' \cong h \Perp f_{p/q}$) have a period two cluster cycle. 
\end{prop}

\begin{proof}
  Since $\theta \in (2/3,1/3)$, the set $A_1 = \{ 2^{2k+1}\theta : k = 0,\ldots,q-1 \}$ is a subset of $(1/3,2/3)$ and by definition has $1$-angular rotation number $p/q$ under the map $t \mapsto 4t$. Hence the angles in $A_1$ correspond to the angles of the rays which land on the (unique) $p/q$ rabbit in $\M_4$ whose associated angles lie in $(1/3,2/3)$. Hence there are exactly $q$ angles which satisfy the properties of the proposition and since by Proposition~\ref{p:alldata2}, all $q$ different critical displacements can be obtained, we are done.
\end{proof}

\subsection{Proof of Theorem~\ref{mainthm2}}\label{Proofmainthm2}

We now prove the second main theorem. Again, we require a result from \cite{Thurstoneq}.

\begin{thm}\label{Per2case}[Theorem~B, \cite{Thurstoneq}]
  Suppose that two quadratic rational maps $F$ and $G$ have a period two cluster cycle with rotation number $p/n$ and critical displacement $\delta$. Then $F$ and $G$ are equivalent in the sense of Thurston.
\end{thm}

\begin{proof}[Proof of Theorem~\ref{mainthm2}]
By Proposition~\ref{p:alldata2}, all combinatorial data can be obtained and by Theorem~\ref{Per2case}, maps with the same combinatorial data are Thurston equivalent, so all maps with period two cluster cycles are matings. By Proposition~\ref{p:2clusttunedrabbit}, one of the maps must be either the double rabbit $f_{p/q}$ or the unique period $2q$ polynomial $g_{p/q}$ that lies in the wake of $f_{p/q}$ in $\M$. Finally, by Proposition~\ref{p:hangles} the complementary map $h$ must have an associated angle with 2-angular rotation number $(q-p)/q$. 
\end{proof}

\subsection{Matings with the secondary map}\label{secondmap}

We now focus on the main difference between the matings in the fixed and period two cases, namely that there exist non-trivial shared matings in the period two case, which do not exist in the fixed case. In the fixed case, since one of the maps had to be a However, it is not true that

We describe the general structure of these Hubbard trees of the secondary maps. This can be calculated from algorithms in \cite{SchleichBru}, for example.

\begin{prop}
	The Hubbard tree of a secondary map can be described as follows. There are two period 2 points $p_1$ and $p_2$ with $n$ arms. One of the global arms at $p_1$ contains the critical point $c_0$ and the other global arms at $p_1$ have endpoints $c_1,c_3,\ldots,c_{2n-3}$. The point $p_2$ has one global arm which contains the critical point $c_0$, and the endpoints of the other global arms are $c_2, c_4, \ldots, c_{2n-2}$. The point $c_{2n-1}$ is on the arc $(p_2, c_{2n-1})$. Finally, there are no branch points on the arc $(p_1,p_2)$, and $c_0$ lies on $(p_1,p_2)$.
\end{prop}

As an example, Figure~\ref{f:treepic} shows the Hubbard tree for the secondary map which corresponds to the angles $(407/1023,408/1023)$.

\begin{figure}[htb]
\begin{center}
 
\input{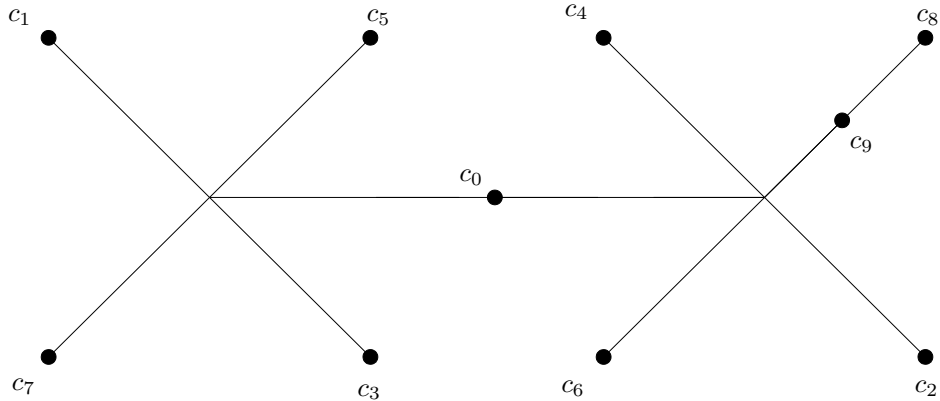}

\caption{Hubbard tree for a secondary map for $q=5$, $\rho = 2/5$. Notice in particular that $c_{2n-1}$ lies on the branch from $p_2$ which has endpoint $c_{2n-2}$.}
\label{f:treepic}
\end{center}
\end{figure}

We now ask which rational maps can be obtained by matings with the secondary map. We will prove the following.

\begin{prop}\label{l:critdisp}
	Suppose $F \cong g_{p/q} \Perp h$ is a rational map. Then the critical displacement of $F$ will be $1$ or $2n-1$. 
\end{prop}

For ease of notation we will drop the subscript on $g_{p/q}$. We will break down the proof into a sequence of lemmas, in which we will show that it is not possible for the rational map constructed by the mating $g \Perp h$ to have any other critical displacement apart from $1$ or $2n-1$. First a comment on the angles of the rays that will be relevant to this discussion. The branch at the period two point $p_1$ in the Hubbard tree containing the critical value of $g$ is separated from the other branches at the period two point by two rays, of angle $\theta$ and $\theta+3$ (we suppress the denominator $2^{2n} -1$). The rays of angles $\theta +1$ and $\theta+2$ land at the root point of the critical value component of $g$, see Figure~\ref{f:critvalrays}.

\begin{figure}[htb]
\begin{center}
 
\input{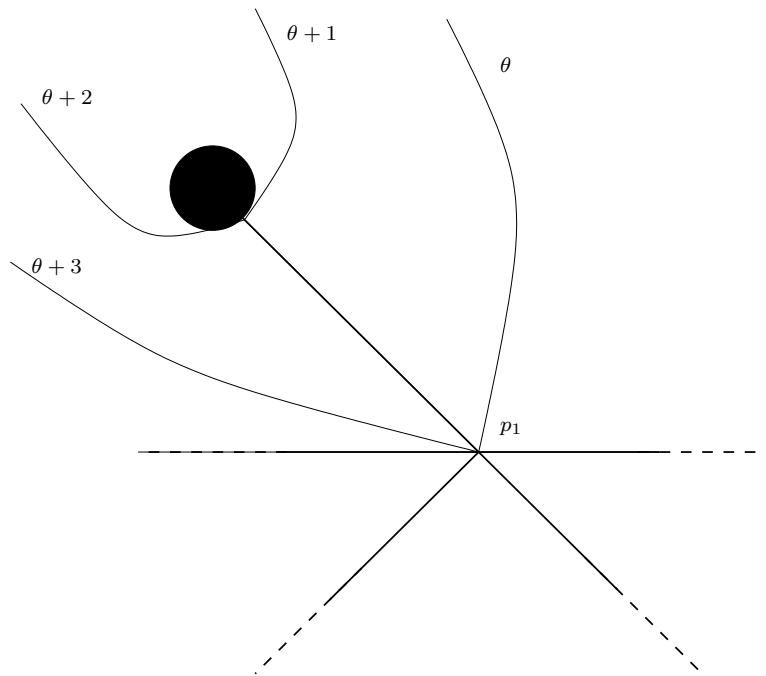}

\caption{The rays landing at the period two point and critical value component of the map $g$.}
\label{f:critvalrays}
\end{center}
\end{figure}

We claim that the ray graph containing the point $p_1$ (as in the picture) must have (precisely) one of the following two properties. The branch of the graph of the ray equivalence class containing the root point $r$ of the critical value component $U_1$ of $g$ (in other words, the landing point of the rays of angle $\theta+1$ and $\theta+2$) must contain either the ray of angle $\theta$ or the ray of angle $\theta+3$. We will show that every other possibility is impossible.

\begin{lem}\label{l1}
 The branch of the ray equivalence class containing $r$ cannot contain the ray $R_{g}(\phi)$, which lands at $p_1$ and such that $\phi \notin \{ \theta, \theta+3 \}$.
\end{lem}

\begin{proof}
It is clear that each branch of the graph must contain exactly one of the rays landing at the point $p_1$. If it contained more than one then they would form a loop and thus would generate a Levy cycle. Suppose the branch containing $r$ contains a ray of angle $\phi \notin \{ \theta, \theta+3 \}$, and $R_g(\phi)$ lands at $p_1$. Let $\gamma$ denote the (unique) path through external rays from $r$ to $p_1$. Since the only possible biaccessible points in $J(g)$ on this path are $r$ and $p_1$, we see that (using $[p,r]$ as the notation for the regulated arc from $p$ to $r$):
\[
	\Gamma = \gamma \cup [p_1,r]
\]
separates the sphere into two pieces. In particular, it separates the root point $r'$ on the branch anticlockwise in the cyclic order from $r$ from the ray of angle $\phi'$ which is the first angle anticlockwise round from $\phi$ in the cyclic ordering of rays at $p$. But since $F$ has to maintain this cyclic ordering, the branch containing $r'$ has to contain the ray of angle $\phi'$, and since rays cannot cross this is a contradiction, see Figure~\ref{f:critpic1}.

\begin{figure}[ht]
\begin{center}

\input{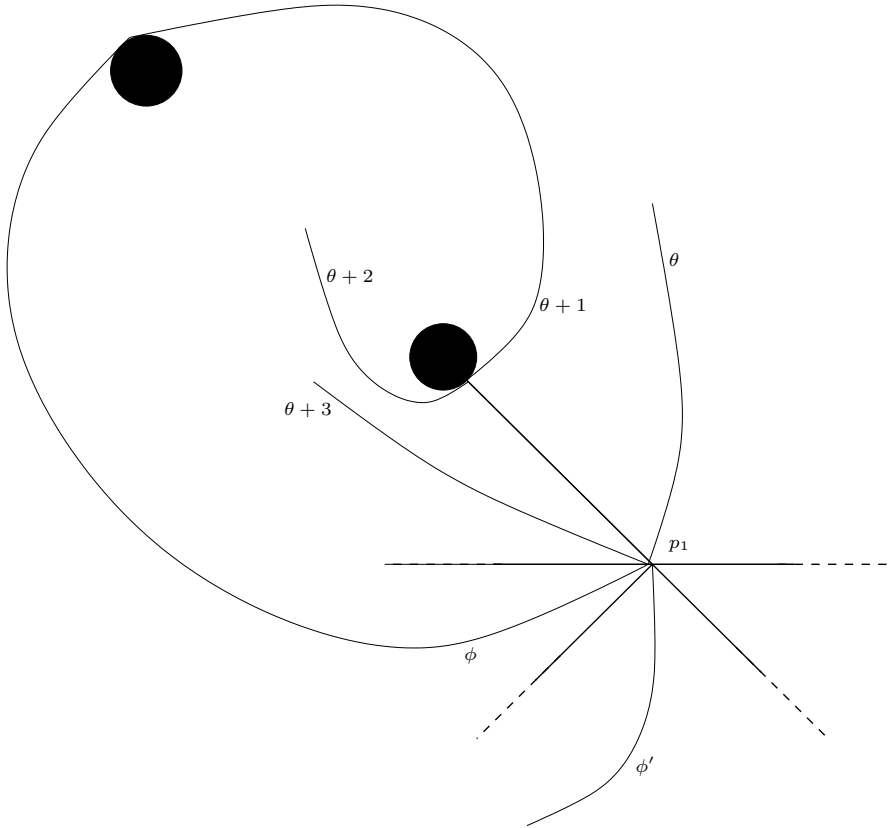}

\caption{The case where the branch containing $r$ contains the ray of angle $\phi$.}
\label{f:critpic1}
\end{center}
\end{figure}	 
\end{proof}

So we now know that the branch of the graph of the ray equivalence class containing $r$ must contain the ray $R_g(\theta)$ or $R_g(\theta+3)$. We now study this branch in more detail.

\begin{lem}\label{l2}
\mbox{}
\begin{enumerate}
  \item{The rays $R_h(-(\theta+1))$ and $R_h(-\theta)$ do not land at the same point on $J(h)$.}
  \item{The rays $R_h(-(\theta+2))$ and $R_h(-(\theta+3))$ do not land at the same point on $J(h)$.}
\end{enumerate}
\end{lem}

\begin{proof}
Suppose that the ray $R_h(-(\theta+1))$ lands at the same point as $R_h(-\theta)$. This common landing point $\zeta$ cannot be the root point of a critical orbit component, since the size of the sector would mean this would have to be a critical value component, and so the critical values would be in the same cluster, a contradiction of Proposition~\ref{p:Marysresult}. However, the angular width of the sector bounded by $\zeta$ and the external rays $R_h(-(\theta+1))$ and $R_h(-\theta)$ is the smallest it can possibly be, and so it must contain the root point of the Fatou component which contains critical point of $h$. But this is another contradiction since the root point of the critical value component needs to be in the ray equivalence class of $f(p_1)$ and it is separated from this point. The proof of when $R_h(-(\theta+2))$ lands at the same point as $R_h(-(\theta+3))$ is analogous to the above.	
\end{proof}

\begin{proof}[Proof of Proposition~\ref{l:critdisp}]
In light of Lemma~\ref{l1} and Lemma~\ref{l2}, the only remaining possibility is that the rays $R_h(-(\theta+2))$ and $R_h(-\theta)$ land at the same point or the rays $R_h(-(\theta+1))$ and $R_h(-(\theta+3))$ land at the same point. Since both situations are essentially the same will we discuss only the first one. The common landing point $\xi$ of the two rays must be the root point of a critical orbit component. For if not, we notice that
\[
	R_g(\theta+2) \cup	R_h(-(\theta+2)) \cup R_h(-\theta) \cup R_g(\theta)
\]
separates the ray $R_h(-(\theta+1))$ from all other rays with the necessary denominator (Figure~\ref{f:critpic2}). Since the root point of a critical orbit component must have (at least) two rays landing on it, this means that this branch of the graph of the ray equivalence class cannot contain a root point of a critical orbit component. But then none of the branches can, as each one maps homeomorphically onto its image and they are periodic. Hence the rays $R_h(-(\theta+2))$ and $R_h(-\theta)$ land at the root point $\hat{r}$ of a critical orbit component. The angular width of the sector containing this component is 2, and so the angular width of its pre-image is 1. This means that the pre-image sector contains the critical value, and so in particular the pre-image of $\hat{r}$ is the root point of the critical value component. Hence $\hat{r}$ is the root point of the component $V_2$ containing the second iterate of the critical point of $h$. Since $V_2$ is adjacent to $U_1$, it follows their pre-images are adjacent, and so the critical point component of $g$ will be adjacent to the critical value component of $h$. This means that the critical displacement of the resultant rational map will have to be $\pm 1$, or equivalently, equal to $1$ or $2n-1$.
\end{proof}

This proof shows us that, in order for a rational map to be a mating with the secondary map as one of the participants, one of the clusters has to see a critical point adjacent to the critical value of the other critical point.

\begin{figure}[ht]
\begin{center}

\input{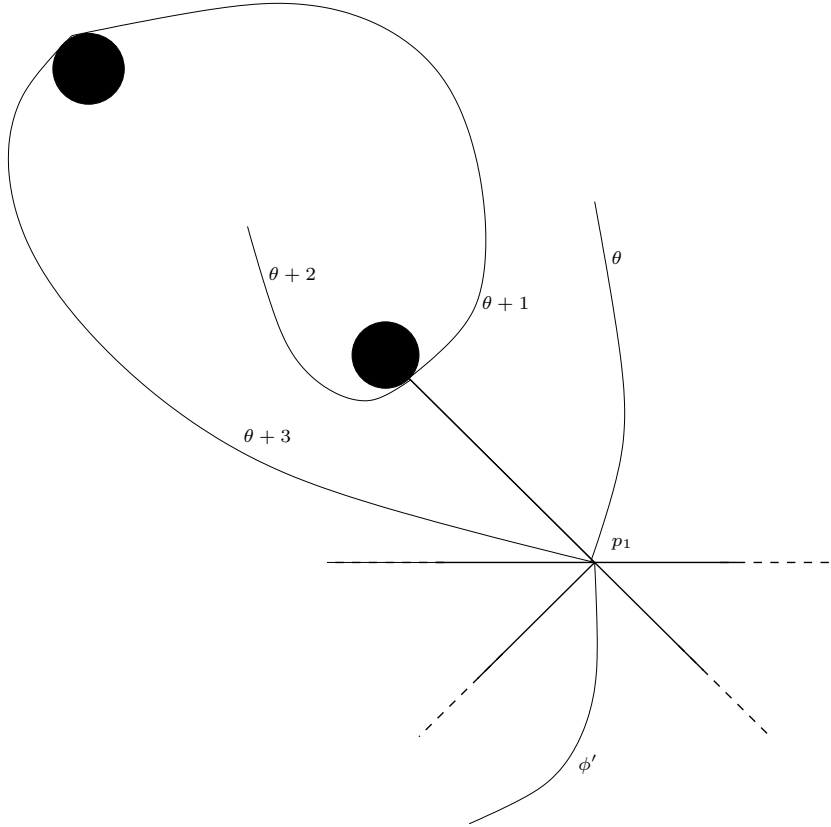}

\caption{How the ray class is formed near the critical value component of $g_{p/q}$.}
\label{f:critpic2}
\end{center}
\end{figure}	

\begin{cor}\label{c:bothmatings}
  Suppose $F \cong g_{p/q} \Perp h$ is a rational map with a period two cluster cycle. Then $F' \cong f_{p/q} \Perp h$ is also a rational map with a period two cluster cycle.
\end{cor}

\begin{proof}
  There are precisely two rays landing at the root points of the critical orbit Fatou components of $h$ and the orbits of these rays are disjoint. Since $F$ is a rational map, the angle $\theta$ associated to $h$ must belong to $(2/3,1/3)$. Furthermore, as $F$ has a period two cluster cycle, one of these ray orbits is made up of angles of the form $-\theta_i$, where $\theta_i$ are the $q$ angles of the rays which land on the period two orbit of $g_{p/q}$. But these are precisely the angles of the rays which land on the period two orbit of $f_{p/q}$. The orbit therefore has $2$-angular rotation number $p/q$ and so one of the angles associated to $h$ satisfies the conditions of Proposition~\ref{p:hanglesconverse}.
\end{proof}

The converse to Corollary~\ref{c:bothmatings} is not true. For example, if $h$ is the quadratic polynomial associated to the angles $(13/63,14/63)$ then $f_{1/3} \Perp h$ is equivalent to a rational map with combinatorial data $(1/3,3)$. However $g_{1/3} \Perp h$ is equivalent to a rational map that does not have a period two cluster cycle. We should now specify exactly how the two critical displacements $\delta = 1$ and $\delta = 2q-1$ are obtained under the mating with the secondary map. Indeed, it is easy to see (for example, by considering ``drawing the rays shut'' in Figure~\ref{f:critpic2}) that if we are in the case where $R_h(-\theta)$ and $R_h(-(\theta+2))$ land at the same point, then $\delta = 1$. Similarly, if we are in the case where $R_h(-(\theta+1))$ and $R_h(-(\theta+3))$ land at the same point, then we will obtain $\delta = 2q-1$ for the resultant rational map. Conversely, if $\delta =1$ (respectively $\delta = 2q-1$) then the rays $R_h(-\theta)$ and $R_h(-(\theta+2))$ (respectively $R_h(-(\theta+1))$ and $R_h(-(\theta+3))$) land at the same point.

%

\section{Shared Matings}\label{s:shared}

A little further work shows that if $g \Perp h$ has critical displacement 1 (respectively $2n-1$), then $f \Perp h$ has critical displacement $2n-1$ (respectively 1). The proof of this fact will allow us to discuss how the matings are shared in the period two cluster case. In the case where the clusters were fixed, the nature of the sharing was relatively simple: one just needs to choose the appropriate rabbit and then mate it with the correct choice of a complementary map to get the required critical displacement (depending on whether the rabbit was chosen to be the first or second map in the mating). In this case, the introduction of the secondary map means the discussion of shared matings has an extra degree of complexity.

\begin{lem}
  $G \cong g_{p/q} \Perp h$ has critical displacement $1$ (respectively $2q-1$) if and only if $F \cong f_{p/q} \Perp h$ has critical displacement $2q-1$ (respectively $1$).
\end{lem}

\begin{proof}
By the observation at the end of the previous section, we see that if $G \cong g_{p/q} \Perp h$ has critical displacement $1$ then the rays $R^h_{-\theta}$ and $R^h_{-(\theta+2)}$ land at the same point on $J(h)$, namely the root point of $U$, the Fatou component which contains $h^{\circ 2}(0)$. So when we carry out the mating $f_{p/q} \Perp h$, we see that since $R^h_{-\theta}$ lands on the root point of $U$, the component $U$ will lie clockwise of the critical value component of the critical point belonging to $f_{p/q}$. This means that the critical displacement will be $2q-1$, as required. The case for critical displacement $2q-1$ is similar.

Conversely, by the fact that the branches of the ray equivalence classes $[p_1]$ and $[p_2]$ are all homeomorphic, it only is necessary to show that if $F \cong f_{p/q} \Perp h$ has a critical displacement of $1$, then the ray class for $g_{p/q} \Perp h$ is as in Figure~\ref{f:critpic2}. By assumption, the ray $R_h(-(\theta +3))$ must land on the root point $r$ of the critical orbit component of $h$ containing $h^{\circ 2}(0)$, which we denote $V_2$. We must show that the partner ray is $R_{h}(-(\theta+1))$. So let the partner ray be $R_h(\phi)$. Clearly, $\phi \neq -(\theta+2)$ since then $V_2$ would have to be the critical value component and $\phi \neq -\theta$ since this would create a loop in the ray equivalence class, meaning the mating would be obstructed. The union $R_{h}(-(\theta+3)) \cup R_{h}(\phi)$ separates $h^{\circ 2}(0)$ from the rest of the critical orbit of $h$ (this is easy to see by, for example, considering the Hubbard tree of $h$). In particular, if $\phi > -(\theta+1)$ then we must have $-(\theta+3) < -(\theta + 1) < \phi$, meaning the ray $R_{h}(-(\theta+1))$ will not land on the root point of a critical orbit Fatou component of $h$. But this would contradict $F$ having a period two cluster cycle, so we see that $\phi = -(\theta+1)$. In particular this means that the ray class is as in Figure~\ref{f:critpic2} and we are done.
\end{proof}

We now have all the information we need to start enumerating examples, save for a simple calculation as was promised after the definition of the period two critical displacement. 

\begin{lem}
  Suppose the bicritical rational map $F$ with labelled critical points has a period two cluster cycle with critical data $(p/q,\delta)$. Then under the alternative labelling it has combinatorial data $(p/q,2p-\delta)$.
\end{lem}

\begin{proof}
  The critical displacement is by definition the combinatorial distance between $c_1$ and $F(c_2)$. By considering pre-images, this is the same as the distance between $F^{\circ (2q-1)}(c_1)$ and $c_2$. Since the combinatorial rotation number is $p/q$, we see that the distance between $F(c_1)$ and $c_2$ is  $\delta - 2p$. Hence the critical displacement when $c_2$ is taken to be the first critical point is $2p-\delta$.
\end{proof}

 We see now that it is possible to get 2-fold, 3-fold and even 4-fold shared matings in the period two case. For example, we have the following equivalences of matings in the period $8$, rotation number $1/4$ case. The notation is as follows:
\begin{align*}
  &f = f_{1/4} \text{ is the double rabbit corresponding to the angles $86/255$ and $89/255$.}\\
  &g = g_{1/4} \text{ is the secondary map corresponding to the angles $87/255$ and $88/255$.}\\
  &h_1 \text{ is the map corresponding to the angles $83/255$ and $84/255$.}\\
  &h_3 \text{ is the map corresponding to the angles $77/255$ and $78/255$.}\\
  &h_5 \text{ is the map corresponding to the angles $53/255$ and $54/255$.}\\
  &h_7 \text{ is the map corresponding to the angles $211/255$ and $212/255$.}
\end{align*}

\noindent With this notation, we have the following equivalences:
\begin{align*}
  &(\delta = 1) 	\quad	\quad	f \Perp h_1 \cong g \Perp h_7 \cong h_1 \Perp f \cong h_7 \Perp g  \\
  &(\delta = 3) 	\quad	\quad	f \Perp h_3 \cong h_7 \Perp f \cong h_1 \Perp g  \\
  &(\delta = 5) 	\quad	\quad	f \Perp h_5 \cong h_5 \Perp f   \\
  &(\delta = 7)		\quad	\quad	f \Perp h_7 \cong g \Perp h_1 \cong h_3 \Perp f
\end{align*}
Indeed, generally the smallest multiplicity of sharing is $2$, and the greatest is $4$. Using the notation above (so that $h_{\delta}$ is the polynomial such that $F \cong f_{p/q} \Perp h_{\delta}$ has combinatorial data $(\rho,\delta)$) we always have the equivalence
\begin{equation}\label{e:shared}
  f_{p/q} \Perp h_{\delta} \cong F \cong h_{2p-\delta} \Perp f_{p/q}.
\end{equation}
Further, as we saw above, it is also sometimes possible to construct a rational map with a period two cluster cycle using a mating with the map $g_{p/q}$. In some cases (as with the case $\delta = 1$ above), this can give us a shared mating with multiplicity $4$.

\begin{proof}[Proof of Theorem~\ref{t:shares}]
All combinatorial data can be obtained by Proposition~\ref{p:alldata2}. By the observation of equation~(\ref{e:shared}), each set of data can be obtained in at least two ways and as stated above, this data can be obtained with a mating with the map $g_{p/q}$ in at most two more ways. As shown in the calculation of the case where $\rho = 1/4$ , all the possible multiplicities can be obtained. 
\end{proof}

\section{Open questions}\label{s:openprobs}

We now ask how we may generalise the results of this paper. The first question is, even if we restrict ourselves to the quadratic case, what can we say about the matings that produce a map with a period $p$ cluster cycle. Some preliminary calculations are in \cite{Mythesis} and Saul Schleimer (\emph{personal communication}) has suggested  that Thurston equivalence for the period three case should be possible. It is always possible to create such a rational map with the mating with an ``order $p$'' rabbit (the higher order analogues to the rabbits and ``double rabbits'' of this paper), but there are in general $2^{p-1}-1$ ``secondary maps''in the period $p$ case, each of which having a period $p$ orbit with a well-defined combinatorial rotation number, and so being a viable candidate for a partipant in a mating that will create a rational map with a period $p$ cluster cycle. Such considerations may lead to a better understanding of parameter space: some considerations of the fixed case (and using parabolic, not hyperbolic maps) can be found in \cite{AdamBuffEcalle}. Moreover, the question remains how else can we create such rational maps and, assuming there are other ways, how prevalent are the shared matings in the higher period cases? Of course, any approach towards answering this question needs to resolve two questions. Firstly, what matings can admit a map with a period $p$ cluster cycle? Secondly, what is the analogue of the Thurston equivalence results as discussed in \cite{Thurstoneq}? Of course, an attempt to understand the higher degree cases would also be an interesting pursuit, since, as this and the sister paper show, even in the relatively benign period two case, the complications can cause even Thurston equivalence to fail in higher degrees.

\vspace{12pt}
\noindent\emph{Acknowledgements.} I would like to thank my PhD supervisor, Dr.~Adam Epstein, for all the help he has contributed towards my research and, in particular, this article. Furthermore, my thanks to Professors Mary Rees and Anthony Manning for providing a number of useful suggestions in the preparation of the manuscript. This research was funded by a grant from EPSRC.

\bibliographystyle{amsalpha}
\bibliography{papers}

\end{document}